\documentclass[12pt]{article}

\usepackage[utf8]{inputenc}
\usepackage{amsmath}
\usepackage{amssymb}
\usepackage{geometry}
\usepackage{verbatim}
\usepackage{bbm}
\usepackage{amsthm}
\usepackage{hyperref}
\usepackage{xcolor}

\def\er{{\mathbb R}}
\def\Pr{{\mathbb P}}
\def\Ex{{\mathbb E}}
\def\ve{\varepsilon}
\def\Var{\mathrm{Var}}

\DeclareMathOperator*{\argmax}{arg\,max}
\newcommand{\mc}[1]{\mathcal{{#1}}}
\newcommand{\1}{\mathbf{{1}}}

\newtheorem{theorem}{Theorem}[section]
\newtheorem{lemma}[theorem]{Lemma}
\newtheorem{fact}[theorem]{Fact}
\newtheorem*{remark}{Remark}
\newtheorem{corollary}[theorem]{Corollary}
\theoremstyle{definition}

\begin{document}

       \title {Comparing moments of real log-concave random variables}

       \author{Daniel Murawski \thanks{The author was supported by the National Science Centre, Poland, grant 2018/31/D/ST1/01355.}}

    \maketitle

\begin{abstract}
We show that for every mean zero log-concave real random variable $X$ one has $\|X\|_p \leq \frac{p}{q} \|X\|_q$ for $p \geq q \geq 1$, going beyond the well-known case of symmetric random variables. We also prove that in the class of arbitrary log-concave real random variables for $p>q > 0$ the quantity $\|X\|_p / \|X\|_q$ is maximized for some shifted exponential distribution. Building upon this we  derive the bound  $\|X\|_p \leq C_0 \frac{p}{q} \|X\|_q$ for arbitrary log-concave $X$, with best possible absolute constant $C_0=e^{W(1/e)} \approx 1.3211$ in front of $\frac{p}{q}$, where $W$ stands for the Lambert function. 
\end{abstract}

\section{Introduction}

The study of moments of random variables plays an important role in probability theory. In particular, in concentration of measure theory, convex geometry and in the probabilistic Banach spaces theory one is often interested in moment comparison inequalities, that is bounds of the form $\|X\|_p \leq C_{p,q,X} \|X\|_q$ with $p \geq q >0$, see \cite{BGVV}, \cite{L01} and \cite{LT11}. Here for a real random variable $X$ we define $\|X\|_p= (\Ex |X|^p)^{1/p}$ and if $X$ instead of being real has values in a certain Banach space, the absolute value in the above definition has to be replaced with the corresponding Banach space norm. As an example let us mention the famous Khintchine inequality, in which case one considers $X$ being a sum of independent symmetric two-point random variables, leading to a constant $C_{p,q}$ independent on the number of summands, see   \cite{K23}.

Clearly it is not possible to derive a universal bound of the form  $\|X\|_p \leq C_{p,q} \|X\|_q$ for $p>q>0$ if one considers the  class of  all real random variables. However, such an inequality is true in the class of all positive random variables with log-concave tails, that is, for which the function $t \mapsto \Pr(X > t)$ is log-concave. Barlow, Marshall and Proschan proved  the following theorem, see \cite{BMP}.

\begin{theorem}\label{thm:sym}
If $X$ is a positive or symmetric real random variable with log-concave
tails and $E$ has exponential distribution with parameter 1, then the function $p \mapsto \|X\|_p / \|E\|_p$ is nonincreasing on $(0,\infty)$, that is, for $p \geq q > 0$ we have $\|X\|_p \leq C_{p,q} \|X\|_q$ with $C_{p,q}=\|E\|_p / \|E\|_q$, where $\|E\|_p = \Gamma(p+1)^{1/p}$.
\end{theorem}

A real random variables $X$ is called log-concave if $X$ is constant a.s. or $X$ has a density of the form $e^{-V}$, where $V:\er \to \er \cup \{+\infty\}$ is convex \cite{Borell}. Since due to Pr\'ekopa-Leindler inequality every log-concave random variable has log-concave tails, the above theorem is valid also in the class of log-concave symmetric (or positive) real random variables. The constant $C_{p,q}$ is of order $p/q$ for large $p,q$. In fact one always has $C_{p,q} \leq p/q$, see Corollary \ref{gamma}. In this article we show that the inequality $\|X\|_p \leq \frac{p}{q} \|X\|_q$ holds true not only for symmetric log-concave random variables, but also for random variables having mean zero.

\begin{theorem}\label{thm:zeroMean}
Let $X$ be a log-concave real random variable with mean zero. Then for $p \geq q \geq 1$ we have $\|X\|_p \leq \frac{p}{q} \|X\|_q$. 
\end{theorem} 

It is an interesting open problem to find the best constant $C_{p,q}$ in the above inequality in the class of mean zero log-concave random variables. In \cite{E21} Eitan showed that for an arbitrary positive even integer $n$ one has $C_{n,2} = \left(n! \sum_{k=0}^n (-1)^k /k!  \right)^{1/n}$. The inequality is saturated for $X=\Gamma$, where $\Gamma$ has density $e^{-(x+1)}\1_{[-1,\infty)}$. In fact the author showed that for any $p \geq q \geq 1$ and for any mean zero log-concave random variable $X$ one has
\[
	\frac{\|X\|_p}{\|X\|_q} \leq \max_{0 \leq s \leq 1} \frac{\|\Gamma^s\|_p}{\|\Gamma^s\|_q},
\]
where $\Gamma^s = s \Gamma - (1-s) \Gamma'$ and $\Gamma'$ is an independent copy of $\Gamma$. The author conjectured that $s \in \{0,1\}$ always achieves the above maximum and verified this conjecture for $q<p<100$. 

Let us also mention that in fact for symmetric log-concave random variables a more general version of Theorem \ref{thm:sym}. is known. Namely, there is a  description of densities maximizing and minimizing $\|X\|_{p_{n+1}}$ under  fixed moments $\|X\|_{p_k}$, $k=1,\ldots, n$ for arbitrary $n$ and arbitrary $p_1,\ldots, p_{n+1}>-1$, see \cite{ENT18}.   

We now focus on arbitrary not necessarily symmetric or centered log-concave random variables. We show that in this case the best constant $C_{p,q}$ is achieved by a certain shifted exponential distribution $X_0 = E+a$, where $a$ is a real number.     

\begin{theorem}\label{thm:shifted-exp}
If $\mc{P}$ denotes the class of all log-concave random variables, then for $p > q > 0$ we have
\[
\sup_{X \in \mc{P}} \frac{\|X\|_p}{\|X\|_q} = \frac{\|X_0\|_p}{\|X_0\|_q},
\] 
where $X_0$ is a shifted exponential distribution.
\end{theorem}

\noindent The proof of this theorem uses an important concept of \emph{degrees of freedom} of a log-concave function, developed by Fradelizi and Gu\'edon in \cite{FraGue}. For applications of this method in the context of  entropy see \cite{MNT21, BN21}.   

Let $W$ be the Lambert function (inverse of $xe^x$) which is well-defined as a function $W:\er_+ \to \er_+$. Using Theorem \ref{thm:shifted-exp}. we show the following result.

\begin{theorem}\label{thm:bestConst}
For any log-concave real random variable X and any $p > q  \geq 2$ the inequality $\|X\|_p \leq C_0 \frac{p}{q} \|X\|_q$ holds true, where $C_0=e^{W(1/e)} \approx 1.3211$. Moreover, this is
the least such constant independent of $p,q$. 
\end{theorem}

The article is organized as follows. In Section \ref{section:Stirling} we analyze an exact Stirling's formula for $\Gamma$ function. In Section \ref{section:zeroMean}. we prove the Theorem \ref{thm:zeroMean}. In Section \ref{section:generalCase} we first prove the Theorem \ref{thm:shifted-exp} and then use it to prove the Theorem \ref{thm:bestConst}.
\vspace{1 cm}

\noindent\textbf{Acknowledgments}

The author would like to thank prof. Rafa{\l} Lata{\l}a for his mentorship and for providing many results shown in this article.
The author would also like to thank dr Piotr Nayar for numerous helpful discussions and suggesting the degrees of freedom approach as a good starting point for the case of general logarithmically concave random variable.

\section{Stirling's formula}\label{section:Stirling}

Throughout the paper we will often make use of the following form of Stirling's formula (Theorem 1.6.3 in \cite{AAR}).
\begin{theorem}[Stirling's formula for $\Gamma$ function]
We have
$$\Gamma(x+1) = \Big(\frac xe\Big)^x \sqrt{2 \pi x} e^{\mu(x)},$$
where $\mu(x)$ is given by the formula $\mu(x) = \int_0^{\infty} \arctan\big(\frac tx\big)(e^{2\pi t} - 1)^{-1} dt$. In particular, $\mu(x)$ is a decreasing function of $x$ and $\mu(x) \geq 0$.
\end{theorem}

From this formula, the corollary below follows easily.
\begin{corollary}\label{gamma}
The function $f(p) = \frac{\Gamma(p+1)^{1/p}}{p}$ is decreasing on $[1, \infty]$
\end{corollary}
\begin{proof}
We have
$$\frac{\Gamma(p+1)^{1/p}}{p} = \frac{\frac pe \sqrt[2p]{ 2 \pi p}e^{\mu(p)/p}}{p} = \frac 1e \Big(2\pi p\Big)^{1/2p} e^{\mu(p)/p}.$$
The function $e^{\mu(p)/p}$ is decreasing. For $p \geq 1$ we also have
$$\frac d{dp} \ln\Big(\big(2\pi p\big)^{1/2p}\Big) = \frac d{dp} \frac1{2p} \ln(2\pi p) = \frac{1-\ln(2\pi p)}{2p^2} \leq \frac{1-\ln(2\pi )}{2p^2} < 0.$$
Therefore the function $\Big(2\pi p\Big)^{1/2p}$ is also decreasing, from which the statement follows.
\end{proof}

This corollary lets us observe that $C_{p, q} \leq \frac pq$ for symmetric or nonnegative log-concave random variables.
\begin{corollary}\label{porSym}
If $X$ is a log-concave random variable which is symmetric or nonnegative, then for $p > q \geq 1$
$$\|X\|_p \leq \frac pq \|X\|_q.$$
\end{corollary}
\begin{proof}
Apply Corollary \ref{gamma}. to Theorem \ref{thm:sym}.
\end{proof}
The Corollary \ref{gamma}. and Corollary \ref{porSym}. have been first proven by R. Lata{\l}a and J. O. Wojtaszczyk \cite{LW}.

\section{Log-concave random variables with mean 0.}\label{section:zeroMean}

In this section we will prove the theorem \ref{thm:zeroMean}.
If $X = 0$ almost surely, then the inequality is trivial. By classification of log-concave measures \cite{Borell} we can focus on variables with continuous distributions. \\
\bigskip

\noindent\textbf{Notation. }
We will denote $\Pr(X < 0)$ by $r$. Without loss of generality, we might assume that $r \leq \frac 12$ (otherwise we can analyse $-X$).

\bigskip

\noindent The following theorem gives a bound for $r$ from below:
\begin{theorem}[Gr\"unbaum's Inequality]\label{Grunbaum}
If $f: \mathbb{R} \to [0, \infty)$ is an integrable, log-concave function such that $\int_{-\infty}^{\infty} xf(x) dx = 0$, then
$$\int_{-\infty}^{0} f(x) dx \geq e^{-1}\int_{-\infty}^{\infty} f(x) dx.$$
\end{theorem}
Proof of this theorem can be found in \cite{Grunbaum}.

\begin{corollary}
If $X$ is a log-concave random variable with mean 0, then
$r = \Pr(X < 0) \geq e^{-1}$.
\end{corollary}
\begin{proof}
Let $g$ be the density of $X$. Then $g$ is a nonnegative log-concave function and we have
$$\int_{-\infty}^\infty xg(x) dx = \Ex X = 0.$$
By Gr\"unbaum's inequality we have
$$r = \Pr(X < 0) = \int_{-\infty}^{0} g(x) dx \geq e^{-1} \int_{-\infty}^{\infty} g(x) dx = e^{-1}.$$
\end{proof}

First we will show Theorem \ref{thm:zeroMean}. for sufficiently large $q$.

\begin{lemma}
If $X$ is a zero-mean log-concave real random variable, then Theorem \ref{thm:zeroMean}. is true for $p > q \geq \frac{e^3}{2\pi}$.
\end{lemma}

\begin{proof}

Let $X_1, X_2$ be random variables with the same distributions as $X$ conditioned on $\{X<0\}$, $\{X\geq 0\}$ respectively (so they have densities $\frac{1}{r}g_X(x)I_{\{x<0\}}$,
$\frac{1}{1-r}g_X(x)I_{\{x>0\}})$). 
Then
$\Ex X^s=r\Ex X_1^s+(1-r)\Ex X_2^s$ for all $s>0$. 

Fix $p>q$ and to shorten the notation, let
$$a_i:=\|X_i\|_p,\ b_i:=\|X_i\|_q,\ c_s:=\Gamma(s+1)^{1/s}.$$
We want to show that
\[
\|X\|_p = (ra_1^p+(1-r)a_2^p)^{1/p}\leq \frac{p}{q}(rb_1^q+(1-r)b_2^q)^{1/q}= \frac pq \|X\|_q.
\]
We know that
\[
\frac{1}{e}\leq r\leq \frac{1}{2},\quad a_i\leq \frac{c_p}{c_q}b_i.
\]
Using this, we have
    $$\|X\|_p=(ra_1^p + (1-r)a_2^p)^{1/p} \leq \frac{c_p}{c_q}(rb_1^p + (1-r)b_2^p)^{1/p}.$$
    Observe that for $\alpha \geq 1$, $t \geq 0$ the following inequality is true
    $$(1 + \alpha t)^{q/p} \leq 1 + (\alpha t)^{q/p} \leq 1 + \alpha t^{q/p}.$$
    By applying this inequality to $\alpha = \frac{1-r}r \geq 1$ (because $r \leq \frac 12$) and $t = \big(\frac {b_2}{b_1}\big)^p$ and raising both sides to the power $1/q$ we get
    $$\Big(1+\frac{1-r}r \Big(\frac {b_2}{b_1}\Big)^p\Big)^{1/p} \leq \Big(1 + \frac{1-r}r\Big(\frac {b_2}{b_1}\Big)^q\Big)^{1/q}.$$
    After multiplying by $r^{1/p}b_1$ we have
    $$(rb_1^p + (1-r)b_2^p)^{1/p} \leq r^{1/p-1/q}(rb_1^q + (1-r)b_2^q)^{1/q} \leq e^{1/q-1/p}(rb_1^q + (1-r)b_2^q)^{1/q},$$
    where in the last inequality we used the inequality $r \geq e^{-1}$.
    It remains to check, that $\frac {c_p}{c_q} e^{1/q-1/p} \leq \frac pq$.
    This is equivalent to showing that for $x > \frac{e^3}{2\pi}$ the function
$$f(x):=e^{-1/x} \Gamma(x+1)^{1/x} \frac 1x = \Big(\frac 1e \Gamma(x+1) x^{-x}\Big)^{1/x} $$
is decreasing.

By Theorem \ref{gamma}.,  $\Gamma(x+1) = \sqrt{2\pi}x^{x+1/2}e^{-x} e^{\mu(x)}$, where $\mu$ is decreasing. Therefore
$$f(x) = \Big(\frac 1e \sqrt{2\pi x} e^{-x+\mu(x)}\Big)^{1/x}, \quad
\ln f(x) = \frac{\mu(x)}x + \frac{\ln(2\pi e^{-2}x)}{2x} - 1.$$
We note that $\frac{\mu(x)}x$ is decreasing, so it suffices to check that $\frac{\ln(2\pi e^{-2}x)}{2x}$ is a decreasing function. Its derivative is
$$\frac d{dx} \frac{\ln(2\pi e^{-2}x)}{2x} = - \frac{\ln(2\pi x) - 3}{2x^2}.$$
 If $x \geq \frac {e^3}{2\pi}$, then $\ln(2 \pi x) \geq 3$, which finishes the proof.

\end{proof}

\vspace{0.5 cm}
$ $\\
\textbf{Proof of Theorem \ref{thm:zeroMean}. in the full range of parameters}
\vspace{0.5cm}\\
In the preceding reasoning, the key was to show that the function $(e(rb_1^p + (1-r)b_2^p))^{1/p}$ is decreasing with respect to $p$. Now we will find for each $p$ such $A_p$, that if for all $q$ in $(p-\ve, p)$ the function $(A_p(rb_1^q + (1-r)b_2^q))^{1/q}$ is decreasing, then the Theorem \ref{thm:zeroMean}. follows. 

\begin{lemma}\label{Ap}
If, using the previous notation, for all $p > 1$ there exists $\ve >0$ such that for $1 < p-\ve <q \leq p$ we have
$$(A_p(rb_1^p + (1-r)b_2^p))^{1/p} \leq (A_p(rb_1^q + (1-r)b_2^q))^{1/q},$$
where $A_p = \sqrt{\frac {2\pi p}e},$
then the Theorem \ref{thm:zeroMean}. follows.
\end{lemma}

\begin{proof}
The function $p \mapsto \frac 1p\|X\|_p$ is differentiable. Therefore, if for all $p > 1$ there exists $\varepsilon > 0$ such that Theorem \ref{thm:zeroMean} is true for $p \geq q > p - \varepsilon > 1$, then the inequality is true for all $p >q \geq 1$. That is because this means that the derivative of $p \mapsto \frac 1p\|X\|_p$ is nonpositive on $(0, \infty)$. This lets us examine the inequality locally.\\
Define
$$f_p(x) := A_p^{-\frac 1x} \frac 1x \Gamma(x+1)^{1/x} = A_p^{-1/x}\frac{c_x}x.$$
Then
$$\ln f_p(x) = \frac{\ln(2 \pi A_p^{-2} x)}{2x} + \frac{\mu(x)}{x} - 1.$$
The function $\frac{\mu(x)}{x} - 1$ has a negative and continuous derivative. Moreover,
$$\frac d{dx} \frac{\ln(2 \pi A_p^{-2} x)}{2x}=-\frac 1{2x^2} (\ln(\frac{2 \pi x}{e}) - \ln(A_p^2)),$$
which is less or equal to zero if and only if $A_p \leq \sqrt{\frac{2 \pi x}e}$. Thus, if $A_p = \sqrt{\frac{2\pi p}e}$, then the function $f_p(x)$ is decreasing on $(0, p]$.

If for $p \geq q > p-\ve > 1$ we have
$$(A_p(rb_1^p + (1-r)b_2^p))^{1/p} \leq (A_p(rb_1^q + (1-r)b_2^q))^{1/q},$$
then
$$A_p^{1/p}\|X\|_p \leq \frac{c_p}{c_q}(A_p(rb_1^p + (1-r)b_2^p))^{1/p} \leq \frac{c_p}{c_q}(A_p(rb_1^q + (1-r)b_2^q))^{1/q} = \frac{c_p}{c_q}A_p^{1/q}\|X\|_q,$$
so
$$\frac{A_p^{1/p}}{c_p}\|X\|_p \leq \frac{A_p^{1/q}}{c_q}\|X\|_q.$$
We have already proven that 
$$f_p(p) =A_p^{-1/p}\frac{c_p}p \leq A_p^{-1/q}\frac{c_q}q = f_p(q).$$
Multiplying the last two inequalities we get
$$\frac1p \|X\|_p \leq \frac1q \|X\|_q,$$
which is the statement of Theorem \ref{thm:zeroMean}.
\end{proof}
In light of the previous lemma, it suffices to prove for $0 \leq q \leq p$ that
$$(A_p(rb_1^p + (1-r)b_2^p))^{1/p} \leq (A_p(rb_1^q + (1-r)b_2^q))^{1/q}.$$
Denote $y := \frac{b_2}{b_1}$. The inequality transforms to

$$(A_p(r + (1-r)y^p))^{1/p} \leq (A_p(r + (1-r)y^q))^{1/q}.$$
Thus, we can investigate the monotonicity of $x \mapsto (A_p(r + (1-r)y^x))^{1/x}$ on $(0, p]$. We have
\begin{align*}
 &\frac d{dx} (A_p(r + (1-r)y^x))^{1/x} \\ 
 &=\frac{(A_p(r+(1-r)y^x))^\frac1x}{x^2(r+(1-r)y^x)} \cdot \Big(
(1-r)y^x \ln (y^x) - (r+(1-r)y^x)(\ln A_p + \ln(r+(1-r)y^x))\Big).
\end{align*}

The first factor is always positive, the second one, after transformations and substitution $t:=y^x$ is nonpositive if and only if
$$\frac{t^\frac{(1-r)t}{r+(1-r)t}}{r+(1-r)t} \leq A_p.$$
Therefore, it suffices to show that
$$\frac{t^\frac{(1-r)t}{r+(1-r)t}}{r+(1-r)t} \leq \sqrt{\frac{2 \pi p}{e}}.$$ 
Let us further denote $\alpha := \frac r{1-r}$, then $\alpha \in [\frac 1{e-1}, 1]$. We get

$$\frac{t^\frac{(1-r)t}{r+(1-r)t}}{r+(1-r)t} = 
\frac{t^\frac{t}{t+\alpha}}{t + \alpha} \cdot (\alpha+1) =: W(t, \alpha).$$
We have
$$\frac d{dt} \frac{t^{\frac{t}{t+\alpha}}}{t + \alpha} = 
\frac{\alpha t^{\frac t{t+\alpha}}}{(t+\alpha)^3} \cdot \ln t.$$
The sign of the derivative is the same as the sign of $\ln t$, so  for fixed $\alpha$ the function $t \mapsto W(t, \alpha)$ is decreasing on $(0, 1]$ and increasing on $[1, \infty)$. Recall that $t = y^x$, where $x$ is in the interval $(0, p]$. It follows that if for a fixed $y$ the inequality is true for $x=p$, then it is true for all values of $x$ in that interval. Observe, that by definition of $r$ and the fact that $\Ex X = 0$ we have

$$\alpha \|X_1\|_1 = \|X_2\|_1.$$
By log-concavity and the inequality between norms ($\|X\|_q \leq \|X\|_p$ for $1 \leq q \leq p)$ we get

$$\|X_i\|_1 \leq \|X_i\|_q \leq \|X_i\|_p \leq \Gamma(p+1)^\frac1p \|X_i\|_1 \leq p \|X_i \|_1.$$
Recall, that $t = y^x = \big(\frac{\|X_2\|_q}{\|X_1\|_q}\big)^x$. It is enough to prove the inequality at $x=p$, so we can estimate

$$t = \Big(\frac{\|X_2\|_q}{\|X_1\|_q}\Big)^p \leq \Big(\frac{p\alpha \|X_1\|_1}{\|X_1\|_1}\Big)^p = \alpha^pp^p$$
and
$$t= \Big(\frac{\|X_2\|_q}{\|X_1\|_q}\Big)^p \geq \Big(\frac{\alpha\|X_1\|_1}{\Gamma(p+1)^\frac1p \|X_1\|_1}\Big)^p = \frac{\alpha^p}{\Gamma(p+1)}.$$
Therefore, it remains to prove  the following technical lemma.
\begin{lemma}\label{LemObl}
Suppose that $p \geq 1$, $\alpha \in [\frac{1}{e-1}, 1]$ and $t = \frac{\alpha^p}{\Gamma(p+1)}$ or $t = \alpha^pp^p$ and $\alpha^pp^p \geq 1$.Then
$$W(t, \alpha) = \frac{t^\frac{t}{t+\alpha}}{t+\alpha} \cdot (\alpha + 1) \leq \sqrt{\frac{2\pi p}{e}}.$$
\end{lemma}

In order to prove this lemma, we will use two auxiliary facts.
\begin{fact}\label{faktpom1}
For $p \geq 2$ we have
$$\frac{1}{e-1} \in \argmax_{\alpha \in [\frac1{e-1}, 1]} \frac{1+(\frac{\alpha^{p-1}}{\Gamma(p+1)})^2}{1+\frac{\alpha^{p-1}}{\Gamma(p+1)}}.$$
\end{fact}

\begin{proof}
Let us calculate the derivative of $h(x)=\frac{1+x^2}{1+x}$. We get
$$h'(x) = \frac{x^2+2x-1}{(x+1)^2}=\frac{(x+1+\sqrt{2})(x+1-\sqrt2)}{(x+1)^2}.$$
Therefore, $h$ is decreasing on the interval $(0, \sqrt2-1)$ and increasing on $(\sqrt2-1, \infty)$. We are looking for $\alpha$ maximizing $h(\frac{\alpha^{p-1}}{\Gamma(p+1)})$ and the function $\frac{\alpha^{p-1}}{\Gamma(p+1)}$ is increasing with respect to $\alpha$, so the maximum is attained at one of the ends of the interval $[\frac1{e-1}, 1]$. It is therefore sufficient to show that $h(\frac1{\Gamma(p+1)}) \leq h\Big(\frac{1}{(e-1)^{p-1}\Gamma(p+1)}\Big)$. 
We have
$$\frac{1}{(e-1)^{p-1}\Gamma(p+1)} < \frac1{\Gamma(p+1)}$$
and
$$\frac{1}{(e-1)^{p-1}\Gamma(p+1)} \leq \frac1{2(e-1)} < \sqrt2-1.$$
Thus, it remains to prove the inequality when $\frac1{\Gamma(p+1)} > \sqrt2-1$. Observe that both $\frac{1}{(e-1)^{p-1}\Gamma(p+1)}$ and $\frac1{\Gamma(p+1)}$ are decreasing with respect to $p$. Therefore, if $\frac1{\Gamma(p+1)} > \sqrt2-1$, then the function $h(\frac{1}{(e-1)^{p-1}\Gamma(p+1)})$ increases with $p$ and $h(\frac1{\Gamma(p+1)})$ decreases with $p$. Thus, it suffices to check the inequality for $p=2$. Indeed, $h(\frac1{2(e-1)}) > 0.84$, while $h(\frac12) = \frac 56 < 0.84$.
\end{proof}

\begin{fact}\label{faktpom2}
For $0 < c \leq \frac 13$ we have
$1+c^2 \leq \sqrt{1+c}$.
\end{fact}
\begin{proof}

We want to show that
$1+c^2 \leq \sqrt{1+c}.$
After squaring both sides we get
$1+2c^2+c^4 \leq 1+c.$
For $c > 0$ this is equivalent to
$2c + c^3 \leq 1,$
which is clearly the case for $0 < c \leq \frac13$.
\end{proof}

\begin{proof}[Proof of lemma \ref{LemObl}]
$ $\\
\textbf{Case 1. $t =p^p\alpha^p \geq 1$.}

We have $t \geq 1$, so
 $t^\frac{t}{\alpha + t} \leq t$, therefore

$$\frac{t^\frac{t}{\alpha + t}}{t+\alpha} \leq \frac t{t+\alpha} \leq 1.$$
From this, $W(t, \alpha) \leq \alpha + 1 \leq 2$. This proves the inequality for $p \geq \frac{2e}{\pi}$ (in this case $\sqrt{\frac {2\pi p}e} \geq 2$), it remains to show it for $p \leq \frac{2e}{\pi} < \sqrt 3$. In this case

$$t \leq \sqrt{3} ^{\sqrt3} \alpha ^ {\sqrt3} 
\leq 3\alpha^{\sqrt3} 
\leq 3.$$
The inequality $t^\frac{t}{\alpha + t} \leq t$ implies

$$W(t, \alpha) \leq \frac{t(\alpha+1)}{t+\alpha}.$$
The function $\frac t{t+\alpha}$ is increasing with respect to $t$, $t \leq 3$, so

$$\frac{t(\alpha+1)}{t+\alpha} \leq \frac{3(\alpha+1)}{\alpha+3}.$$
The function $\frac{\alpha+1}{\alpha+3}$ is increasing with respect to $\alpha$, $\alpha \leq 1$, so $\frac{3(\alpha+1)}{\alpha+3} \leq \frac32$. Thus

$$W(t, \alpha)  \leq \frac 32 \leq \sqrt{\frac {2\pi}{e}} \leq \sqrt{\frac{2\pi p}e},$$
which completes the proof of the first case.
\vspace{0.5 cm}\\
\textbf{Case 2. $t =\frac{\alpha^p}{\Gamma(p+1)}\leq 1$, $p \geq 2$.}\\
We have
$$W(t, \alpha) =
t^\frac{t}{t+\alpha} \cdot \frac\alpha{t+\alpha} \cdot \frac{\alpha+1}\alpha.$$
The factor $\frac{\alpha+1}\alpha$ is decreasing with respect to $\alpha$, so

$$\frac{\alpha+1}\alpha \leq \Big(\frac1{e-1}+1\Big)(e-1) =  e.$$
We also note that

$$\frac\alpha{t+\alpha} = \frac1{1+\frac{\alpha^{p-1}}{\Gamma(p+1)}}.$$
This expression decreases with $\alpha$.

By AM-GM inequality
$$t^\frac{t}{t+\alpha} = t^\frac{t}{t+\alpha} \cdot 1^{\frac{\alpha}{t + \alpha}}
\leq \frac{t^2}{t+\alpha} + \frac{\alpha}{t+\alpha} = \frac{t^2 + \alpha}{t+\alpha} =
\frac{1+\frac{\alpha^{2p-1}}{\Gamma(p+1)^2}}{1+\frac{\alpha^{p-1}}{\Gamma(p+1)}} \leq
\frac{1+(\frac{\alpha^{p-1}}{\Gamma(p+1)})^2}{1+\frac{\alpha^{p-1}}{\Gamma(p+1)}} .$$
By the Fact \ref{faktpom1}. it suffices to check $\alpha = \frac1{e-1}$, because it maximizes the last expression. Similarly for $\frac{\alpha+1}{\alpha}$ i $\frac{\alpha}{t+\alpha}$.
Since $\frac{\alpha^{p-1}}{\Gamma(p+1)} \leq \frac1{2(e-1)} < \frac 13$, then by the Fact \ref{faktpom2}.
 $$1+(\frac{\alpha^{p-1}}{\Gamma(p+1)})^2 \leq \sqrt{1+\frac{\alpha^{p-1}}{\Gamma(p+1)}}.$$
 Thus, we get

$$W(t, \alpha) \leq  e\cdot \left(1+\frac{(e-1)^{1-p}}{\Gamma(p+1)}\right)^{-\frac32}.$$
Since we are proving $W(t, \alpha) \leq \sqrt{\frac{2\pi p}{e}}$, then after squaring both sides it remains to show that

$$\left(1+\frac{(e-1)^{1-p}}{\Gamma(p+1)}\right)^{-3} \leq \frac{2\pi p}{e^3}.$$
We estimate
$$\left(1+\frac{(e-1)^{1-p}}{\Gamma(p+1)}\right)^{-3} \leq
\left(1+3\frac{(e-1)^{1-p}}{\Gamma(p+1)}\right)^{-1} = 
\frac{\Gamma(p+1)}{\Gamma(p+1) + 3(e-1)^{1-p}}.$$
After rewriting we get the following inequality to show

$$\Gamma(p)\Big(\frac{e^3}{2\pi} - p\Big) \leq  3(e-1)^{1-p}.$$
We have $\frac {e^3}{2\pi} < 3.2$, so it suffices to check $2 \leq p \leq 3.2$ .
By convexity of exponential function and estimating by the tangent at $p=3$

$$(e-1)^{1-p} \geq (e-1)^{-2}(1-(p-3)\ln(e-1)).$$
We also have $e-1 \leq \sqrt3$ and $\frac12 \leq \ln(e-1) \leq \frac59$, which gives 

$$3(e-1)^{1-p} \geq 1-\frac12(p-3)$$
for $p \leq 3$ and

$$3(e-1)^{1-p} \geq 1-\frac59(p-3)$$
for $p \geq 3$.
By convexity of $\Gamma$ we get $\Gamma(p) \leq p-1$ for $p \in [2, 3]$ and $\Gamma(p) \leq 4p-10$ for $p \in [3, 4]$. Therefore it suffices to check two quadratic inequalities

$$(p-1)(\frac{16}5 - p) \leq 1 - \frac12(p-3)\textrm{, where }p \in[2, 3].$$
$$(4p-10)(\frac{16}5-p) \leq 1 - \frac59(p-3) \textrm{, where }p\in[3, 3.2].$$
Those inequalities are true.
\vspace{0.5 cm}\\
\textbf{Case 3. $t = \frac{\alpha^p}{\Gamma(p+1)} \leq 1$, $p \in [1, 2]$.}

 By convexity of $\Gamma$ for $p \in [1, 2]$ we get $\Gamma(p+1) \leq p$. Denote $t_0 := \frac{\alpha^p}p < t \leq 1$. Thus, $W(t_0, \alpha) \geq W(t, \alpha)$, so it suffices to show that $W(t_0, \alpha) \leq \sqrt{\frac{2\pi p}e}$. Recall that $W(t_0, \alpha) = t_0^{\frac{t_0}{t_0+\alpha}} \cdot \frac{\alpha+1}{t_0+\alpha}$. We have

$$t_0^\frac {t_0}{t_0+\alpha} = \Big(\frac{\alpha^p}p\Big)^\frac{\alpha^{p-1}}{\alpha^{p-1} + p}
\leq p^{-\frac{\alpha^{p-1}}{\alpha^{p-1} + p}}$$
and
$$\frac{\alpha^{p-1}}{\alpha^{p-1} + p} \geq \frac{(e-1)^{-1}}{(e-1)^{-1}+2}
= \frac 1{2e-1} \geq \frac15,$$
so $t_0^\frac {t_0}{t_0+\alpha} \leq p^{-1/5}$, $W(t_0, \alpha) \leq p^{-1/5} \frac{\alpha+1}{t_0+\alpha}$. Therefore, we want to show that (after dividing by $p^{1/2}$)

\begin{equation}\label{pom7}
    \frac{\alpha+1}{p^{7/10}(t_0+\alpha)} \leq \sqrt{\frac{2\pi}e}.
\end{equation}
For this purpose, we will investigate the expression

$$\frac{\alpha+1}{p^{7/10}(t_0+\alpha)} = \frac{\alpha+1}{\alpha} \cdot \frac1{p^{7/10}(\frac{\alpha^{p-1}}p + 1)}.$$
For $p \geq 1$ both factors are nonincreasing with respect to $\alpha$, so we can assume $\alpha = \frac1{e-1}$. After substituting $\alpha= \frac{1}{e-1}$ to the previous inequality and transformations, we get that the inequality \eqref{pom7} is equivalent to

$$p^{-3/10}(e-1)^{1-p} + p^{7/10} \geq \sqrt{\frac{e^3}{2\pi}}.$$
We have bounds

$$(e-1)^{1-p} \geq 1 - \ln(e-1)(p-1) \geq 1 - \frac59(p-1),$$
$$p^{-3/10} \geq 1 - \frac3{10}(p-1),$$
$$p^{7/10} \geq 1 + \frac7{10}(p-1) - \frac{21}{200}(p-1)^2,$$
by expanding corresponding functions into Taylor series at $1$. Introducing the variable $x = p-1$ we have to check the following inequality for $x \in [0, q]$.

$$\left(1-\frac59 x\right)\left(1-\frac3{10}x\right) + 1 + \frac7{10}x - \frac{21}{200}x^2 \geq \sqrt{\frac{e^3}{2\pi}}.$$
We have
$$\left(1-\frac59 x\right)\left(1-\frac3{10}x\right) + 1 + \frac7{10}x - \frac{21}{200}x^2 = \frac{37}{600}x^2-\frac{7}{45}x+2 \geq 2-\frac15x \geq 1.8.$$
However, $\sqrt{\frac{e^3}{2\pi}} < 1.8$, so the inequality is true.
\end{proof}

\section{Arbitrary real log-concave random variables}\label{section:generalCase}

In this section we will first prove the Theorem \ref{thm:shifted-exp}. and then use it to prove the Theorem \ref{thm:bestConst}.

\subsection{Proof of Theorem \ref{thm:shifted-exp}.}

In order to prove the Theorem \ref{thm:shifted-exp}., we will use results by M. Fradelizi and O. Gu\'edon from \cite{FraGue} (Chapter 3).

\begin{theorem} \label{fradelizi}
Let $\mathcal P$ be the set of log-concave probability distributions on the segment $[a,b]$. Let $P_f \subset \mathcal P$ be the set of such distributions $\mu$ that $\int f_i d\mu \geq 0$ for some fixed, continuous $f_1, ..., f_p$ that are linearly independent. If $\Phi$ is convex and continuous on $\mathcal P$, then $\sup \{\Phi(\mu) : \mu \in P_f\}$ is reached for a distribution $\nu$ with density $e^{-V}$ such that
\begin{itemize}
    \item for $k \geq 1$ functions among $f_1, ..., f_p$ we have $\int f_i d\nu = 0$,
    \item $V = \sup\{\phi_1, ..., \phi_k\}$ for some affine functions $\phi_1, ..., \phi_k$.
\end{itemize}

\end{theorem}

\begin{remark}
We use Theorem 2 from 
\cite{FraGue}
by setting $K = [a, b]$, $G = \mathbb R$, $d = 1$. The original theorem is much more general.
\end{remark}

Fix $p > q > 0$. Let $g = e^{-V}$, where $V$ is a convex function, be the density of the random variable $X$. We are investigating the ratio of norms so without loss of generality $\|X\|_q = 1$. We know that $g$ is a probability density and we have the following condition
$$\int_{-\infty}^{\infty}|x|^q g(x) dx = 1.$$
Under this condition, we want to maximize the expression
$$\int_{-\infty}^{\infty}|x|^p g(x) dx.$$
By scaling properties, it is equivalent to maximizing under the condition
$$\int_{-\infty}^{\infty}-(|x|^q-1)g(x) dx \geq 0.$$

Assume that $X$ has a compact support, so we have a distribution on some segment. We can assume that this segment is not contained in either $[0, \infty)$ or $(-\infty, 0]$, because by Theorem \ref{thm:sym}. for variables of one sign we have
$$\frac{\|X\|_p}{\|X\|_q} \leq \frac{\Gamma(p+1)^{1/p}}{\Gamma(q+1)^{1/q}}$$
with equality for exponential distributions. Thus, we can assume that the support is  $[-a, b]$ for some $a, b > 0$. Now we are in a situation of Theorem \ref{fradelizi}., with one equation 
$$\int_{-a}^b (|x|^q-1) g(x) dx = 0$$
and we want to maximize
$$\Phi(\mu) = \int_{-a}^b |x|^p d\mu.$$
Thus, the maximum is achieved for $g = e^{-V}$ where $V$ is affine.  We have
$$g(x) = \mathbbm{1}_{[-a, b]}(x)e^{\alpha x + \beta}.$$
Now we also let $a, b$ vary and we are looking for supremum over variables with compact support. There are four parameters: $a, b, \alpha, \beta$ and two equations
$$1 = \int_{-a}^b e^{\alpha x + \beta} dx =:
F_1(a, b, \alpha, \beta),$$
$$1 = \int_{-a}^b |x|^q e^{\alpha x + \beta} dx =:
F_2(a, b, \alpha, \beta).$$
Under these conditions we want to maximize
$$\int_{-a}^b |x|^p e^{\alpha x + \beta} dx =: G(a, b, \alpha, \beta).$$
The functions $F_1, F_2$ and $G$ are smooth, so this is maximization of a $C^1$ function over differentiable manifold.

\begin{lemma}
In the situation above, the only critical points are symmetric uniform distributions.
\end{lemma}

\begin{proof}
We will use the method of Lagrange multipliers. If $(a, b, \alpha, \beta)$ is a critical point, then 
$$DG(a, b, \alpha, \beta) - \lambda_1DF_1(a, b, \alpha, \beta) -
\lambda_2DF_2(a, b, \alpha, \beta) = 0$$
for some $\lambda_1, \lambda_2$.
After calculating the corresponding partial derivatives and simplifying we get
\begin{equation} \label{eq:Lagrange1}
    a^p - \lambda_2a^q - \lambda_1 = 0, \\
\end{equation}
\begin{equation} \label{eq:Lagrange2}
     b^p - \lambda_2b^q - \lambda_1 = 0, \\
\end{equation}
\begin{equation} \label{eq:Lagrange3}
       \int_{-a}^be^{\alpha x}(|x|^p - \lambda_2|x|^q - \lambda_1)dx = 0, \\
\end{equation}
\begin{equation} \label{eq:Lagrange4}
    \int_{-a}^bxe^{\alpha x}(|x|^p - \lambda_2|x|^q - \lambda_1)dx = 0.
\end{equation}

First, suppose that $a \neq b$, without loss of generality $a<b$, then the first two equations are linearly independent as equations for $\lambda_1, \lambda_2$, so they uniquely determine $\lambda_1, \lambda_2$. Let $r(x) = |x|^p - \lambda_2|x|^q - \lambda_1$. It is an even function with zeroes in $a, b$. The function $r(x^{1/q})$ is strictly convex for $x \geq 0$ and has two zeroes $a^q, b^q$, so it is negative on $(a^q, b^q)$ and positive on $[0, +\infty) \setminus [a^q, b^q]$. Therefore, for $|x| <a$, we have $r(x) > 0$ and for $|x| \in (a, b)$ we have $r(x) < 0$. Thus, we get that
$(a-x)r(x) > 0$ almost everywhere on $(-a, b)$, so
$$a \int_{-a}^b e^{\alpha x} r(x) dx > \int_{-a}^b xe^{\alpha x} r(x) dx$$
which, by equations (\ref{eq:Lagrange3}), (\ref{eq:Lagrange4}) and $a > 0$ leads to $0 > 0$, a contradiction.

Now, suppose that $b = a$. Fix some $\lambda_1, \lambda_2$ satisfying $a^p - \lambda_2 a^q - \lambda_1$ and define $r(x)$ as before. Since $\int_{-a}^a e^{\alpha x}r(x) dx = 0$ and $r$
 is even, $r$ has to change sign somewhere on $(0, a)$, so there is a $c \in (0, a)$, such that $r(c) = 0$. We argue as before, that $r$ is positive on $(0, c)$ and negative on $(c, a)$.
 
Suppose that $\alpha \neq 0$, without loss of generality $\alpha > 0$, $(a, a, \alpha, \beta)$ is a critical point. We have
$$\int_0^a r(x)(e^{\alpha x} + e^{- \alpha x}) dx = 0, \quad \int_0^a xr(x)(e^{\alpha x} - e^{- \alpha x}) dx = 0.$$
For $x \in [0, a]$ we have $r(x)(x-c) \leq 0$, so
$$xr(x)(e^{\alpha x} - e^{- \alpha x}) \leq cr(x)(e^{\alpha x} - e^{- \alpha x})$$
and the inequality is strict almost everywhere, so
$$0 = \int_0^a xr(x)(e^{\alpha x} - e^{- \alpha x}) dx <
c \int_0^ar(x)(e^{\alpha x} - e^{- \alpha x}) dx.$$
We also know that $\int_0^a e^{\alpha x} r(x) dx = - \int_0^a e^{-\alpha x} r(x) dx$, so $\int_0^a e^{\alpha x} r(x) > 0$, $\int_0^a e^{-\alpha x} r(x) dx < 0$. Note that $e^{2\alpha c - \alpha x}r(x) \geq e^{\alpha x} r(x)$ - this is because $(e^{2\alpha c} - e^{2\alpha x})r(x) \geq 0$. Therefore
$$0 > e^{2\alpha c} \int_0^a e^{-\alpha x} r(x) dx = \int_0^a e^{2\alpha c- \alpha x}r(x) dx \geq \int_0^a e^{\alpha x} r(x) dx > 0.$$
This is a contradiction. We have shown that if $(a, b, \alpha, \beta)$ is a critical point, then it corresponds to symmetric uniform distribution. 
\end{proof}

By Theorem \ref{thm:sym}. for symmetric variables we have $\frac{\|X\|_p}{\|X\|_q} \leq \frac{\|E\|_p}{\|E\|_q}$, where $E$ has exponential distribution, so we can focus on the limiting cases.

\bigskip
\noindent
\textbf{Observation. } Let $X$ be a random variable with density $\mathbbm 1_{[-a, b]}e^{\alpha x + \beta}$ and $\|X\|_q = 1$, $\alpha \geq 0$. Then there is a constant $M_q$ which depends only on $q$ such that $b \leq M_q$.\\
\begin{proof}
Let $M_q$ be a unique positive real number such that
$M_q^{q+1} - (q+1)M_q - q = 0$. This is well defined because the function $x^{q+1} - (q+1)x - q$ is convex, has negative value at $x = 0$ and tends to infinity with $x$ going to infinity.
Suppose that $\alpha \geq 0$, $b > M_q$. We have
$$1 = \|X\|_q^q = \int_{-a}^b|x|^q e^{\alpha x + \beta}dx =
1 + e^{\beta}\int_{-a}^b \big(|x|^q-1\big) e^{\alpha x} dx.$$
Observe, that the function $|x|^q-1$ takes negative values in $(-1, 1)$ and positive values outside of $[-1, 1]$. Thus,
$$0 = \int_{-a}^b \big(|x|^q-1\big) e^{\alpha x} dx > 
\int_{-1}^{M_q} \big(|x|^q-1)e^{\alpha x} dx.$$
For $x \geq -1$ we have $e^{\alpha x}\big(|x|^q-1\big) \geq e^{\alpha}\big(|x|^q-1\big)$. From this
$$0 > \int_{-1}^{M_q} \big(|x|^q-1)e^{\alpha x} dx \geq e^{\alpha}\int_{-1}^{M_q} \big(|x|^q-1\big) dx.$$
But by simple calculation
$$\int_{-1}^{M_q} \big( |x|^q - 1\big) dx =  \frac1{q+1}\left(M_q^{q+1} - (q+1)M_q - q\right) = 0.$$
This is a contradiction.
\end{proof}

Denote the set of random variables with densities of the form $\mathbbm 1_{[-a, b]}e^{\alpha x + \beta}$ for $a, b > 0$ by $\mathcal F$. Observe that because $\mathbbm 1_{[-a, b]}e^{\alpha x + \beta}$ is a probability density, we get that $e^{-\beta} = \int_{-a}^b e^{\alpha x}$, so $\beta$ is a function of $\alpha, a, b$. Therefore the set $\mathcal F$ can be parametrized in $A = \mathbb R \times (0, \infty)^2 \subset \mathbb R^3$ in such a way, that the point $(\alpha, a, b)$ corresponds to the variable with density $\mathbbm 1_{[-a, b]}e^{\alpha x + \beta}$, where $\beta$ is uniquely determined. For $y=(\alpha, a, b)$, denote the random variable corresponding to $y$ by $X(y)$. Then, $\|X(y)\|_s$ is a differiantiable function in $A$ for any $s > 0$. Let $(y_n)_{n=1}^{\infty}$, $y_n=(\alpha_n, a_n, b_n)$ be such a sequence of points in $A$, that for each $n$ we have  $\|X(y_n)\|_q = 1$ (denote this set by $A_1$, by previous considerations it is a differentiable manifold). Suppose also that
$$\lim_{n \to \infty}\|X(y_n)\|_p = \sup_{\substack{X \in \mathcal F\\ \|X\|_q = 1}} \|X\|_p.$$
If the sequence $y_n$ is bounded, then it contains a subsequence convergent to some $y_0$. The point $y_0$ has to be a critical point of  $y \mapsto \|X(y)\|_p$ in $A_1$. We have shown that such points correspond to symmetrical uniform distributions which can be taken out of consideration. If the sequence is not bounded, then it contains a subsequence $y_{n_i}$, for which some fixed coordinate tends to $\pm \infty$. This subsequence contains an infinite subsequence such that for all $y_n$ in this subsequence $\alpha_n \geq 0$ or for all of them $\alpha_n \leq 0$. By symmetry, without loss of generality for all $i$ we have $\alpha_{n_i} \geq 0$. Thus, the sequence $b_{n_i}$ is bounded (contained in $[0, M_q]$), so $\alpha_{n_i}$ tends to infinity or $a_{n_i}$ tends to infinity. If $a_{n_i}$ is bounded, then it contains a subsequence convergent to some $a_0$ and because $\alpha_{n_i}$ tends to infinity, we get a subsequence of $X(y_{n_i})$ convergent in distribution to $\delta_{a_0}$, for which the supremum of $\|X(y)\|_p$ over $y \in A_1$ is obviously not achieved. 
 
By passing to subsequences, the only remaining case is that of sequence $(y_n)_{n=1}^{\infty}$, $y_n=(\alpha_n, a_n, b_n)$, where $\alpha_n, b_n$ are convergent and $a_n$ tends to infinity.
Then $X(y_n)$ converges in distribution to $-E + b$, where $b$ is limit of $b_n$ and $E$ is exponential distribution with parameter $\alpha = \lim_{n \to \infty} \alpha_n$. By symmetry and homogeneity of $\frac{\|X\|_p}{\|X\|_q}$ we get the statement of Theorem \ref{thm:shifted-exp}.

\subsection{Analysis of shifted exponential distributions}

Let $X$ have exponential distribution with parameter $1$, we will look at variables of the form $X - t$, where $t$ is a real number. The goal of this section is to prove that for any $t$ and $p > q \geq 2$ we have $$\|X-t\|_p \leq C_0 \frac pq \|X-t\|_q,$$
where $C_0 = e^{W(1/e)}$. If $t < 0$, then we can apply the Corollary \ref{porSym}., so from here on we assume that $t \geq 0$.
By $m_s(t)$ we will denote the $s$-th moment of such a random variable, so
$$m_s(t):=\Ex(E-t)^s = \int_{-t}^{\infty} e^{-x-t}|x|^s dx = 
\int_0^t e^{x-t}x^s dx + e^{-t}\Gamma(s+1) =: I(s, t) + e^{-t}\Gamma(s+1).$$
We can see that
$$I(s, t) = \int_0^t e^{x-t}x^s dx \leq t^s \int_0^t e^{x-t} = (1-e^{-t})t^s.$$

\begin{lemma}[Inequality for $\frac pq \geq 2$]\label{duzy iloraz}
If $\frac pq \geq 2$, $q \geq 2$ and $t \geq 0$, then $\frac{\|X-t\|_p}{\|X-t\|_q} \leq C_0 \frac pq$.
\end{lemma}

\begin{proof}
We can bound the $p-$th moment of $X-t$ from above in the following way:
$$m_p(t)=\Ex(X-t)^p \leq (1-e^{-t})t^p + e^{-t}\Gamma(p+1).$$
Observe that this is a convex combination of $t^p$ and $\Gamma(p+1)$, so
$$(\Ex(X-t)^p)^{1/p} \leq \max(t^p, \Gamma(p+1))^{1/p} = \max(t, \Gamma(p+1)^{1/p}).$$
First, we will show that the statement of this lemma is true, if the maximum is $t$. We have
$$\|X-t\|_q \geq \|X-t\|_2 = \sqrt{\Var(X) + (t-\Ex(X))^2} = ((t-1)^2 + 1)^{1/2} \geq \frac t{\sqrt 2}.$$
The last inequality can be proved by QM-AM inequality. From this, if $t \geq \Gamma(p+1)^{1/p}$, then
$$\frac{\|X-t\|_p}{\|X-t\|_q} \leq \sqrt 2 = (\sqrt 2 \frac qp) \frac pq \leq \frac 1{\sqrt 2} \frac pq < C_0 \frac pq.$$
Thus, it suffices to show that under conditions of the lemma the following is true
$$\Gamma(p+1)^{1/p} \leq C_0\frac pq \|X-t\|_q.$$
We have
$$m_q(t) = e^{-t}\int_{0}^{t} e^{x}x^q dx + e^{-t}\Gamma(q+1),$$
$$m_q'(t) = -m_q(t) + t^q,$$
$$m_q''(t) = m_q(t) - t^q + qt^{q-1} = qt^{q-1} - m'(t).$$
The derivative at zero is negative, so $m_q(t)$ is decreasing on some interval $[0, t_0]$, $u > 0$. If $m_q'(t) = 0$, then $m_q''(t) > 0$, so $m_q(t)$ cannot have a local maximum. If $m_q(t)$ had two local minima, then there would be a local maximum between them, so it can have at most one local minimum.
On the other hand, if $t^q > \Gamma(q+1)$, then $m_q(t) \leq \max(t^q, \Gamma(q+1)) \leq t^q$, which means that $m_q'(t) > 0$ and the function is increasing. This shows that there exists exactly one $t_q$, such that $m_q'(t_q) = 0$ and it is a global minimum of $m_q$. Moreover, because $m_q'(t_q) = 0$, we get $m_q(t_q) = t_q^q$, so it is sufficient to show that 
\begin{equation}\label{pom1}
   C_0 \frac pq t_q \geq \Gamma(p+1)^{1/p}. 
\end{equation}

By $u_q$ we denote the solution of equation $u^q = e^{-u}\Gamma(q+1)$. This equation has a unique solution -- at $u=0$ the left hand side is less than the right hand side, the left hand side increases to infinity with $u$ increasing and the right hand side is decreasing. Observe that $t_q^q = m_q(t_q) \geq e^{-t_q}\Gamma(q+1)$, so by previously observed monotonicities we get $t_q \geq u_q$.

Similarly, we can see that $u_q \geq qW(1/e)$. It is enough to show that $(qW(1/e))^q \leq e^{-qW(1/e)}\Gamma(q+1)$, which can be rewritten into $\big( \frac qe \big)^q \leq \Gamma(q+1)$ and this inequality is true.

Let us introduce the notation $\Delta = \Delta(q) = u_q - qW(1/e) \geq 0$. By Stirling's formula and the definition of $u_q$ we get
$$(qW(1/e))^q \Big(1 + \frac{\Delta}{qW(1/e)}\Big)^q = e^{-\Delta}e^{-qW(1/e)}\Big(\frac qe\Big)^q \sqrt{2 \pi q} e^{\mu(q)}.$$
We can see that $e^{-qW(1/e)} = (eW(1/e))^q$, so
\begin{equation}\label{rown delta}
e^{\Delta}\Big(1 + \frac{\Delta}{qW(1/e)}\Big)^q = \sqrt{2 \pi q} e^{\mu(q)}.  
\end{equation}
Now we will show that $\sqrt{2 \pi q} e^{\mu(q)} \leq 2^q$. Equivalently,  $\sqrt {2\pi} 2^{-q}\sqrt q e^{\mu(q)} \leq 1$. The function $2^{-q}\sqrt q$ is decreasing for $q \geq 2$ (by calculating the derivative), $e^{\mu(q)}$ is decreasing, so it remains to check the inequality at $q = 2$. We have
$$\sqrt {2\pi \cdot 2} e^{\mu(2)} = \frac{\Gamma(3)}{(\frac 2e)^2} = \frac{e^2}2 < 4 = 2^2.$$
This means that 
$$\Big(1+\frac{\Delta}{qW(1/e)}\Big)^q \leq e^{\Delta}\Big(1+\frac{\Delta}{qW(1/e)}\Big)^q \leq 2^q = (1+1)^q,$$
so $\frac{\Delta}{qW(1/e)} \leq 1$. By convexity of the exponential function, and the fact that $1+x \geq e^{x\ln 2 }$ for $x \in \{ 0, 1\}$, we get $1+\frac{\Delta}{qW(1/e)} \geq e^{ \frac {\Delta}{qW(1/e)}\ln 2}$.
This gives the following estimate
$$\Big(1 + \frac{\Delta}{qW(1/e)}\Big)^q \geq e^{ q\frac {\Delta}{qW(1/e)}\ln 2} =
 e^{\Delta \frac {\ln 2}{W(1/e)}} \geq e^{2 \Delta},$$
 because $\frac{\ln 2}{W(1/e)} \approx 2.48 > 2$. Therefore
 $$e^{\Delta}\Big(1 + \frac{\Delta}{qW(1/e)}\Big)^q \leq \Big(1 + \frac{\Delta}{qW(1/e)}\Big)^{\frac 32 q},$$
 so from an earlier equality (\ref{rown delta}) it follows that
 \begin{equation}\label{pom2}
    1 + \frac{\Delta}{qW(1/e)} \geq \Big(\sqrt{2\pi q} e^{\mu(q)}\Big)^{\frac 2{3q}}. 
 \end{equation}
We have
 $$C_0 \frac pq t_q \geq C_0 \frac pq u_q = C_0 \frac pq qW(1/e)\Big(1+\frac{\Delta}{qW(1/e)}\Big),$$
 so in order to finish the proof of \eqref{pom1}, it is enough to show that the last expression is not less than $\Gamma(p+1)^{1/p}$. After raising both sides to the $p-$th power we get the following inequality to show:
 $$\Big(C_0 \frac pq qW(1/e)\Big(1+\frac{\Delta}{qW(1/e)}\Big)\Big)^p \geq \Gamma(p+1).$$
 We see that $C_0 = \frac 1{eW(1/e)}$, therefore
 $$\Big(C_0 \frac pq qW(1/e)\Big)^p = \Big(\frac pe\Big)^p,$$
 so it remains to show that
 $$\Big(1+\frac{\Delta}{qW(1/e)}\Big)^p \geq \sqrt{2 \pi p} e^{\mu(p)}.$$
 By the inequality \eqref{pom2}, it is sufficient to prove
 $$\Big(\sqrt{2\pi q} e^{\mu(q)}\Big)^{\frac {2p}{3q}} \geq  \sqrt{2 \pi p} e^{\mu(p)}.$$
 Since $\frac pq \geq 2 > \frac 32$, then by monotonicity of $\mu$ we have $e^{\frac 23 \frac pq\mu(q)} \geq e^{\mu(q)} \geq e^{\mu(p)}$. Thus, it is enough to show $\Big(\sqrt{2\pi q}\Big)^{\frac {2p}{3q}} \geq  \sqrt{2 \pi p}$, which is the same as $\Big(\sqrt{2\pi q}\Big)^{\frac {2p}{3q}-1} \geq \sqrt{\frac pq}$. Denote $\lambda = \frac pq \geq 2$, now we use the condition $q \geq 2$ and our goal transforms to
 $$\sqrt{4\pi}^{\frac 23 \lambda -1} \geq \sqrt \lambda.$$
 After taking logarithms of both sides:
 $$\Big(\frac 23 \lambda - 1\Big)\ln(4\pi) \geq \ln \lambda.$$
 For $\lambda \geq 2$ the derivative of the right side is not greater than $\frac 12$, while the derivative of the left side is $\frac 23 \ln 4\pi > \frac 12$,  so it suffices to check the inequality for $\lambda = 2$. At $\lambda = 2$ the left hand side is $\approx 0.84$, while the right hand side is $\approx 0.69$, which proves the statement of the lemma.
\end{proof}
$ $\\
Recall that $I(s, t) = \int_0^t e^{x-t}x^s dx$.

\begin{lemma} \label{szacowanie I}
Fix $t$. Then the function $t^{-s}I(s, t)$ decreases with $s$ for $s > 0$.
\end{lemma}
\begin{proof}
We have
$$t^{-s}I(s, t) = \int_0^t e^{x-t}\Big(\frac xt\Big)^s dx.$$
The expression inside the integral is decreasing with respect to $s$.
\end{proof}

\begin{lemma} \label{szacowanie max}
If $X$ has exponential distribution with parameter $1$ and $1 \leq q \leq p \leq 2q$, then
$$\frac{\|X-t\|_p}{\|X-t\|_q} \leq \frac pq \cdot \max((1-e^{-t})^{-1/2q}, e^{t/2q}).$$
\end{lemma}
\begin{proof}
We have
$$\frac{\|X-t\|_p}{\|X-t\|_q} = \frac{\Big(I(p, t) + e^{-t}\Gamma(p+1)\Big)^{1/p}}{\Big(I(q, t) + e^{-t}\Gamma(q+1)\Big)^{1/q}}.$$
By the triangle inequality in $l_{p/q}$ we get
$$\Big(I(p, t) + e^{-t}\Gamma(p+1)\Big)^{q/p} \leq I(p, t)^{q/p} + e^{-tq/p}\Gamma(p+1)^{q/p},$$
thus
$$\frac{\Big(I(p, t) + e^{-t}\Gamma(p+1)\Big)^{1/p}}{\Big(I(q, t) + e^{-t}\Gamma(q+1)\Big)^{1/q}}
\leq \Big(\frac{I(p, t)^{q/p} + e^{-tq/p}\Gamma(p+1)^{q/p}}{I(q, t) + e^{-t}\Gamma(q+1)}\Big)^{1/q}.$$
Now we will use the fact that for positive $a, b, c, d$ we have $\frac{a+b}{c+d} \leq \max(\frac ac, \frac bd)$. From this we can conclude that
$$\frac{\|X-t\|_p}{\|X-t\|_q} \leq \max\Big(\frac{I(p,t)^{1/p}}{I(q, t)^{1/q}}, e^{t/q-t/p}\frac{\Gamma(p+1)^{1/p}}{\Gamma(q+1)^{1/q}}\Big).$$
Observe that since $\frac pq \leq 2$, then $e^{t/q-t/p} \leq e^{t/2q}$. Moreover, from Corrolary \ref{gamma} $\frac{\Gamma(p+1)^{1/p}}{\Gamma(q+1)^{1/q}} \leq \frac pq$, so if the maximum is the second expression, then the lemma is true.

Now, consider a random variable $Y$ with density
$(1-e^{-t})^{-1}e^{x-t}\mathbbm{1}_{[0, t]}$. It is a nonnegative log-concave random variable, so $\|Y\|_p \leq \frac pq \|Y\|_q$. We can see that $\|Y\|_s = ((1-e^{-t})^{-1}I(s, t))^{1/s}$, so
$$I(p, t)^{1/p} \leq \frac pq (1-e^{-t})^{1/p-1/q}I(q, t)^{1/q}.$$
Therefore $\frac{I(p,t)^{1/p}}{I(q, t)^{1/q}} \leq \frac pq (1-e^{-t})^{1/p-1/q} \leq \frac pq (1-e^{-t})^{-1/2q}$, which concludes the proof of the lemma.
\end{proof}

\bigskip
\noindent
\textbf{Notation. } Let $t = r \frac qe$ and denote $r_0 = eW(1/e)$. Then $C_0 = 1/r_0 = e^{r_0/e}$.

\begin{corollary}[Inequality for $r \leq 2r_0$]\label{NierMale r}
If $t = r \frac qe$, $2q \geq p \geq q \geq 2$ and $r \leq 2r_0$, then $\frac{\|X-t\|_p}{\|X-t\|_q} \leq C_0 \frac pq$.
\end{corollary}
\begin{proof}
First, assume that $t \geq 0.4$. In this case we can directly use the previous lemma. Since $r \leq 2r_0$, we see that $e^{t/2q} = e^{r/2e} \leq e^{r_0/e} = C_0$. In order to show that for $q \geq 2$ we have $(1-e^{-t})^{-1/2q} \leq C_0$, it is enough to show this at $q=2$. The inequality becomes $e^{-t} \leq 1-C_0^{-4}$, which is the same as $-t \leq \ln(1 - C_0^{-4}) \approx -0.398$. This means that $t \geq 0.4$ is enough.

Now, let us assume that $t \leq 0.4$. By the proof of the Lemma \ref{duzy iloraz}., $m_s'(t) = t^s - m_s(t)$. If $t \leq 0.4$ and $s \geq 2$, then
$$t^s - m_s(t) \leq 0.16 - e^{-t}\Gamma(s+1)^{1/s} \leq 0.16 - e^{-0.4} \sqrt 2 < 0.$$
From this we can conclude, that for $t \leq 0.4$, $s \geq 2$ we have $\|X-t\|_s \leq \|X\|_s$. Thus, using the moments of exponential distribution we have
$$\|X-t\|_p \leq \|X\|_p \leq \frac pq \|X\|_q.$$
On the other hand, $C_0 \|X-t\|_q \geq C_0 e^{-0.4}\Gamma(q+1)^{1/q} = C_0e^{-0.4} \|X\|_q$. It remains to calculate that $C_0e^{-0.4} \geq 1.08 > 1$, which proves the corollary.
\end{proof}

\bigskip
\noindent
\textbf{Notation. }In the remaining part of the article $g(x)$ will denote the following function
$$g(x) := \Gamma(x+1)\left(\frac xe\right)^{-x} = \sqrt{2\pi x} e^{\mu(x)}.$$
Observe that since $\mu$ is decreasing, we have $g(\lambda x) \leq \sqrt \lambda g(x)$ for $x > 0$, $\lambda \geq 1$.

\bigskip
\noindent
\textbf{Observation. } In order to prove the Theorem \ref{thm:bestConst}. it is enough to prove that
$$\Big(\alpha\Big(\frac r\lambda\Big)^{q\lambda} + \sqrt{\lambda}e^{-rq/e}g(q)\Big)^{1/\lambda} \leq C_0^q\Big(\alpha r^q + e^{-rq/e}g(q)\Big)$$
for $\lambda = \frac pq \leq 2$, $\alpha = t^{-q}I(q,t) \geq t^{-p}I(p, t)$, $t = r\frac qe$ and $r \geq 2r_0$.

\begin{proof}
The Theorem \ref{thm:bestConst}. has already been proven in the cases where $\lambda \geq 2$ or $r \leq 2r_0$.
Using the above notation:
\begin{align*}
\frac{\|X-t\|_p}{\|X-t\|_q} =
\frac{\Big(I(p, t) + e^{-t}\Gamma(p+1)\Big)^{1/p}}{\Big(I(q, t) + e^{-t}\Gamma(q+1)\Big)^{1/q}} &\leq \frac{\Big(\alpha t^p + e^{-t}\Gamma(p+1)\Big)^{1/p}}{\Big(\alpha t^q + e^{-t}\Gamma(q+1)\Big)^{1/q}} \\= 
\frac{\Big(\alpha r^p(\frac pe)^p (\frac qp)^p + e^{-t}(\frac pe)^pg(p)\Big)^{1/p}}{\Big(\alpha r^q(\frac qe)^q + e^{-t}(\frac qe)^qg(q)\Big)^{1/q}} &= 
\frac pq \cdot \Bigg(\frac{\Big(\alpha (\frac r{\lambda})^{p} + e^{-t}g(p)\Big)^{1/\lambda}}{\alpha r^q + e^{-t}g(q)}\Bigg)^{1/q}.
\end{align*}

We want to show that this last expression is less than $C_0 \frac pq$, after transformations the inequality becomes
$$\Big(\alpha \Big(\frac r{\lambda}\Big)^{\lambda q} + e^{-rq/e}g(p)\Big)^{1/\lambda} \leq C_0^q \Big(\alpha r^q + e^{-rq/e}g(q)\Big).$$
Using $g(p) = g(\lambda q) \leq \sqrt \lambda g(q)$, we conclude the statement.
\end{proof}

\begin{lemma}\label{Szacowanie alfa}
For $q \geq 2$, $t = \frac {rq}e$, $r \geq 2r_0 = 2eW(1/e)$ we can estimate
$$\alpha = t^{-q}\int_0^t e^{x-t}x^q dx \geq \frac 19.$$
\end{lemma}
\begin{proof}
Observe that $t \geq W(1/e)$. We estimate
$$t^{-q} \int_0^t e^{x-t}x^q dx \geq t^{-q}\int_{t - W(1/e)}^t e^{x-t}x^q dx \geq t^{-q}W(1/e) \inf_{t-W(\frac 1e) \leq x \leq t} e^{x-t}x^q$$
$$ = t^{-q} W(1/e) e^{-W(1/e)} (t - W(1/e))^q = W(1/e)e^{-W(1/e)}\Big(1-\frac{W(1/e)}t\Big)^q.$$
The value of the last expression increases with $t$. We also note that $\frac{W(1/e)}{\frac{2r_0q}{e}} = \frac 1{2q}$. Therefore, for $t \geq \frac{2r_0q}{e}$ we have the estimate
$$W(1/e)e^{-W(1/e)}\Big(1-\frac{W(1/e)}t\Big)^q \geq
W(1/e)e^{-W(1/e)}\Big(1-\frac{1}{2q}\Big)^q
$$
 Observe, that for $x>1$ we have
$$\frac d{dx} \ln\Big(1-\frac 1x\Big)^x = \ln\Big(1-\frac 1x\Big) + \frac{1}{x-1} \geq \frac{-\frac 1x}{1-\frac 1x} + \frac1{x-1} = 0.$$
Here, we have used the fact that $\ln(1+a) \geq \frac a{a+1}$ for $a > -1$. This means that the function $(1-\frac1x)^x$ is increasing for $x >1$. Thus, the same is true for $(1-\frac1x)^{x/2}$, so for $q \geq 2$ (using $x = 2q$) we get
$$W(1/e)e^{-W(1/e)}\Big(1-\frac 1{2q}\Big)^q \geq W(1/e)e^{-W(1/e)}\cdot \frac 9{16} > \frac 19,$$
which verifies the lemma.
\end{proof}

\begin{lemma}\label{Prawa >1}
If $q \geq 2$, $r \geq 2r_0$, then
$$C_0^q\Big(\alpha r^q + e^{-rq/e}g(q)\Big) \geq 1.$$
\end{lemma}
\begin{proof}
First, observe that $C_0^q\Big(\alpha r^q + e^{-rq/e}g(q)\Big) \geq \alpha (C_0r)^q$, $\alpha \geq \frac 19$, so if $(C_0r)^q \geq 9$, then the inequality is true. Recall that $r_0 = \frac 1{C_0}$, so $C_0r \geq 2$ and $q \geq 4$ is enough and if $q \geq 3$, then it is enough if $r \geq \frac{\sqrt[3]9}{C_0}$.

Now we assume $3 \leq q \leq 4$, $r \leq \frac{\sqrt[3]9}{C_0}$.
Remembering that $C_0^{-1} = e^{-W(1/e)} = eW(1/e)$ we get for $r\leq \frac{\sqrt[3]9}{C_0} = \sqrt[3]9 eW(1/e)$ that
$$C_0e^{-r/e} \geq  e^{W(1/e)(1 - \sqrt[3]9)} \geq 0.74 \geq \frac 1{\sqrt 2}.$$
From this, $\Big(C_0e^{-r/e}\Big)^q \geq \frac 14$. Moreover, $g(q) \geq \sqrt{2\pi q} \geq 1$ and for $q \geq 3$ we have $\alpha (C_0r)^q \geq \frac 89$. This means that for $q \geq 3$ and $r \leq \frac{\sqrt[3]9}{C_0}$ we have
$$C_0^q\Big(\alpha r^q + e^{-rq/e}g(q)\Big) \geq \frac 89 + \frac 14 > 1.$$
Now, assume that $2 \leq q \leq 3$. If $C_0r \geq 3$, then $\alpha (C_0r)^q \geq 1$, so  the inequality is true. Now, suppose that $r \leq \frac 3{C_0}$. Then
$$C_0e^{-r/e} \geq e^{W(1/e)}e^{-3W(1/e)} = e^{-2W(1/e)}.$$
Since $q \leq 3$, then  $\Big(C_0e^{-r/e}\Big)^q \geq e^{-6W(1/e)} \geq \frac 16$. Also, $g(q) \geq \sqrt{2 \pi q} \geq 2\sqrt{\pi}$ for $q \geq 2$. This gives $C_0^qe^{-rq/e}g(q) \geq \frac 13 \sqrt{\pi}$. Moreover, $\alpha (C_0r)^q \geq \frac 49$, so it suffices to check that
$\frac 49 + \frac 13 \sqrt{\pi} \geq 1,$ which is true.
\end{proof}

\begin{corollary}\label{Szac r<e}
If $p \geq q \geq 2$, $2r_0 \leq r \leq e$, $1 \leq \lambda \leq 2$, then
$$\Big(\alpha\Big(\frac r\lambda\Big)^{q\lambda} + \sqrt{\lambda}e^{-rq/e}g(q)\Big)^{1/\lambda} \leq C_0^q\Big(\alpha r^q + e^{-rq/e}g(q)\Big).$$
\end{corollary}
\begin{proof}
In the previous lemma we have shown, that the right hand side is not less than $1$. If the left hand side is less than $1$, then the inequality is obviously true. In the other case,
$$\Big(\alpha\Big(\frac r\lambda\Big)^{q\lambda} + \sqrt{\lambda}e^{-rq/e}g(q)\Big)^{1/\lambda} \leq \alpha\Big(\frac r\lambda\Big)^{q\lambda} + \sqrt{\lambda}e^{-rq/e}g(q).$$
Since $q \geq 2$ i $\lambda \leq 2$, we see that $C_0^q \geq C_0^2 > \sqrt 2 \geq \sqrt \lambda$. Also, $\big( \frac r{\lambda}\big)^\lambda$ is decreasing with respect to $\lambda \in [1,2]$, because
$$\frac d{d\lambda} \ln \Big( \frac r{\lambda}\Big)^\lambda = 
\frac d{d\lambda} \Big(\lambda \ln r - \lambda \ln \lambda\Big) =
\ln r - \ln \lambda - 1 \leq -\ln \lambda \leq 0.$$
Thus
$$\alpha\Big(\frac r\lambda\Big)^{q\lambda} + \sqrt{\lambda}e^{-rq/e}g(q) \leq
\alpha r^q + C_0^qe^{-rq/e}g(q) < C_0^q\Big(\alpha r^q + e^{-rq/e}g(q)\Big).$$
\end{proof}

Thus, we have proven the inequality in the case of $r \leq e$, which is the same as $t \leq q$. The remaining case is $t \geq q$.

\begin{lemma} \label{szacowanie 2. moment}
If $X$ has an exponential distribution with parameter $1$, $p \geq q \geq 2$, $t \geq q$ and $\lambda = \frac pq \geq \lambda_0 = \frac{\sqrt 2}{C_0}$, then
$$\|X-t\|_p \leq C_0 \frac pq \|X-t\|_q.$$
$(\lambda_0 \approx 1.07 < 1.1)$.
\end{lemma}

\begin{proof}
Since $q \geq 2$, we see that $\|X-t\|_q \geq \|X-t\|_2 = \sqrt{\Var X + (\Ex X - t)^2} = \sqrt{1 + (t-1)^2}$. Form this, $\|X-t\|_q \geq \sqrt{t^2-2t+2} \geq \sqrt{q^2-2q+2}$.

By reasoning from the proof of Lemma \ref{duzy iloraz}. we know that $\|X-t\|_p \leq \max(t, \Gamma(p+1)^{1/p})$. By Corollary \ref{gamma}, the function $p^{-1}\Gamma(p+1)^{1/p}$ is decreasing, so $\Gamma(p+1)^{1/p} \leq p\frac{\sqrt{\Gamma(3)}}2 = \frac p{\sqrt 2}$. 
By calculating the derivative, we can see that for $x \geq 2$ the function $f(x) = \frac{x^2-2x+2}{x^2}$ is nondecreasing, so for $x \geq 2$ its value is at least $f(2) = \frac 1{2}$.
Therefore, for $q \geq 2$,
$$C_0\frac pq \|X-t\|_q \geq p \frac{\|X-t\|_2}q \geq p \sqrt{\frac{q^2-2q+2}{q^2}} \geq \frac p{\sqrt 2}.$$
This proves the inequality in the case of $\max(t, \Gamma(p+1)^{1/p}) = \Gamma(p+1)^{1/p}$. On the other hand,
$$\frac 1t C_0 \frac pq \|X-t\|_q \geq C_0 \frac pq \frac{\|X-t\|_2}{t} =
C_0 \lambda \sqrt{\frac{t^2-2t+2}{t^2}} \geq \frac{C_0}{\sqrt 2} \lambda.$$which proves the lemma.
\end{proof}

To summarize, we have proven the Theorem \ref{thm:bestConst}.
for $p \geq q \geq 2$, $\lambda = \frac pq \geq 1$ in the following cases
\begin{itemize}
    \item $\frac pq \geq 2$ (Lemma \ref{duzy iloraz}),
    \item $\frac pq \leq 2$, $t \leq q$ (Corollary \ref{NierMale r}, Observation, Lemma \ref{Szac r<e}),
    \item $\frac pq \geq 1.1$, $t \geq q$ (Lemma \ref{szacowanie 2. moment}). 
\end{itemize}
It remains to check the inequality when $1 \leq \lambda = \frac pq \leq 1.1$, $t = r \frac qe \geq q$, $r \geq e \geq 2r_0$. For this purpose, recall the inequality which is enough to prove the Theorem \ref{thm:bestConst}.
$$\Big(\alpha\Big(\frac r\lambda\Big)^{q\lambda} + \sqrt{\lambda}e^{-t}g(q)\Big)^{1/\lambda} \leq C_0^q\Big(\alpha r^q + e^{-t}g(q)\Big).$$
We will show a stronger inequality
\begin{equation}\label{pom3}
    \Big(\alpha\Big(\frac r\lambda\Big)^{q\lambda} + \sqrt{\lambda}e^{-t}g(q)\Big)^{1/\lambda} \leq C_0^q\alpha r^q.
\end{equation}
In order to do this, we will first bound from above the ratio of $\sqrt \lambda e^{-t}g(q)$ to $\alpha \big(\frac r\lambda\big)^{q\lambda}$. Using the estimates $\alpha \geq \frac19$, $r \geq e$ i $t \geq q$ we get
\begin{equation}\label{pom5}
\frac{\sqrt \lambda e^{-t}g(q)}{ \alpha \big(\frac r\lambda\big)^{q\lambda}} \leq 9 \sqrt \lambda e^{-q}g(q) \Big(e^{-\lambda} \lambda^{\lambda}\Big)^q.
\end{equation}
The function $e^{-q}\sqrt q$ decreases with $q$ for $q \in [2, \infty) $, so $e^{-q}g(q)$ also decreases with $q$ on this interval. Moreover, $e^{-\lambda}\lambda^\lambda < 1$, so the value of the right hand side of \eqref{pom5} decreases with $q$ and we can estimate by substituting $q = 2$. On the other hand, $\frac d{d\lambda} \ln(e^{-\lambda}\lambda^\lambda) = \ln \lambda \geq 0$, $\sqrt \lambda$ increases with  $\lambda$, so we can estimate substituting $\lambda = 1.1$
$$9 \sqrt \lambda e^{-q}g(q) \Big(e^{-\lambda} \lambda^{\lambda}\Big)^q \leq 9\sqrt{1.1} e^{-2}g(2) \Big(e^{-1.1} 1.1^{1.1}\Big)^2 < 0.65.$$
Therefore,
\begin{equation}\label{pom4}
    \Big(\alpha\Big(\frac r\lambda\Big)^{q\lambda} + \sqrt{\lambda}e^{-t}g(q)\Big)^{1/\lambda} \leq \Big(1.65\alpha\Big(\frac r\lambda\Big)^{q\lambda}\Big)^{1/\lambda} = (1.65\alpha)^{1/\lambda}r^q \lambda^{-q}.
\end{equation}
We will show that the last expression is decreasing with respect to $\lambda$. Indeed,
$$\frac d{d\lambda} (1.65\alpha)^{1/\lambda}\lambda^{-q} = (1.65\alpha)^{1/\lambda}\lambda^{-q-2}\big(-q\lambda-\ln(1.65\alpha)\big).$$
We assume that $q \geq 2$, $\lambda \geq 1$, so $-q\lambda \leq -2$. On the other hand, $\alpha \geq \frac 19$, so $1.65\alpha \geq 1.65/9 > e^{-2}$ and $-\ln(1.65\alpha) < 2$. From this, the derivative is negative, so the expression is indeed decreasing with respect to $\lambda$ and we can estimate by setting $\lambda=1$. Then, by \eqref{pom4}
$$\Big(\alpha\Big(\frac r\lambda\Big)^{q\lambda} + \sqrt{\lambda}e^{-t}g(q)\Big)^{1/\lambda} \leq 1.65\alpha r^q.$$
But $C_0^2 > 1.65$, so
$$1.65 \alpha r^q < C_0^2 \alpha r^q \leq C_0^q \alpha r^q.$$
From this, we proved \eqref{pom3}, which finishes the proof of the Theorem \ref{thm:bestConst}. It remains to see that $C_0$ is the best possible constant. This is true by the following lemma.

\begin{lemma}\label{Szacowanie z dolsu}
For any constant $C < C_0$ there exist $p\geq q\geq 2$, $t > 0$, such that $\|X-t\|_p > C\frac pq \|X-t\|_q$.
\end{lemma}
\begin{proof}

Let $t = \frac {r_0}e q$. Then $t^q = e^{-t} \big(\frac qe\big)^q$. We have 
$$\Ex|X-t|^q= I(q, t) + e^{-t}\Gamma(q+1) \leq t^q + e^{-t}\Gamma(q+1)
= \Big(\frac qe\Big)^q e^{-t} (1+ g(q)).$$
Thus,
$$\|X-t\|_q \leq \frac qe e^{-\frac {r_0}e}(1+g(q))^{1/q} = \frac qe \frac 1{C_0}\Big(1+g(q)\Big)^{1/q}.$$
 Because $1 + g(q) = O(\sqrt q)$, we see that $(1 + g(q))^{1/q}$ tends to $1$ with $q \rightarrow \infty$. This means that for any $\varepsilon > 0$ for sufficiently large $q$ we can bound this expression from above by $1 + \varepsilon$. We will also use the estimate $\|X-t\|_p \geq e^{-t/p}\Gamma(p+1)^{1/p} \geq e^{-t/p} \frac pe$. With fixed $t$ for any $\varepsilon > 0$ for sufficiently large $p$ we have $e^{-t/p} \geq 1 - \varepsilon$. Thus, for sufficiently large $p \gg q \gg 2$
$$\frac{\|X-t\|_p}{\|X-t\|_q} \geq C_0 \frac pq \cdot \frac{1-\varepsilon}{1+\varepsilon}.$$
Finally, we see that for any $C<C_0$ for sufficiently small $\varepsilon$ we have $\frac{1-\varepsilon}{1+\varepsilon} C_0 \geq C$.

\end{proof}

\noindent  Faculty of Mathematics, Informatics and Mechanics \\ 
University of Warsaw \\ 
Banacha 2, 02-097, Warsaw, Poland \\
email: dk.murawski@student.uw.edu.pl

\end{document}